\documentclass[12pt]{article}
\usepackage{amsmath,amssymb,amsfonts,amsthm}
\usepackage[margin=1in]{geometry}
\usepackage{graphicx,subcaption}

\theoremstyle{plain}

\newtheorem{theorem}{Theorem}[section]
\newtheorem{lemma}[theorem]{Lemma}
\newtheorem{proposition}[theorem]{Proposition}

\title{Face-simple minimal quadrangulations of surfaces}
\author{Sarah Abusaif, Warren Singh, and Timothy Sun\\Department of Computer Science\\San Francisco State University}
\date{}

\begin{document}

\maketitle

\begin{abstract}
For each surface besides the sphere, projective plane, and Klein bottle, we construct a face-simple minimal quadrangulation, i.e., a simple quadrangulation on the fewest number of vertices possible, whose dual is also a simple graph. Our result answers a question of Liu, Ellingham, and Ye while providing a simpler proof of their main result. The inductive construction is based on an earlier idea for finding near-quadrangular embeddings of the complete graphs using the diamond sum operation. 
\end{abstract}

\section{Introduction}
A \emph{quadrangular} embedding of a graph is a cellular embedding where every face has length 4. A \emph{minimal quadrangulation} of a surface is a quadrangular embedding of a simple graph which has the smallest number of vertices out of all simple graphs that quadrangulate that surface. The problem of finding a minimal quadrangulation for each closed surface was proposed by Hartsfield and Ringel \cite{Hartsfield-Nonorientable, Hartsfield-Orientable}, who solved some special cases using quadrangular embeddings of complete graphs and octahedral graphs.

Let $\phi\colon G \to S$ be a cellular embedding of a graph $G = (V,E)$ in a surface $S$ of Euler characteristic $\chi(S)$, and let $F$ denote the set of faces of $\phi$. The Euler polyhedral formula states that 
$$|V| - |E| + |F| = \chi(S).$$ 
If $\phi$ is quadrangular, then substituting $4|F| = 2|E|$ results in the constraint
$$|E| = 2|V|-2\chi(S).$$ 
From this equation, a necessary condition for a graph $G = (V,E)$ to have a quadrangular embedding is that $|E| \equiv 2|V| \pmod{2}$ for nonorientable surfaces, and $|E| \equiv 2|V| \pmod{4}$ for orientable surfaces. 

Hartsfield and Ringel's problem can be rephrased in terms of finding quadrangular embeddings of graphs on a certain number of vertices and edges. We say that an embedding is an \emph{$(n,t)$-quadrangulation} if it is a quadrangular embedding of a simple graph on $n$ vertices and $\binom{n}{2}-t$ edges. By the above necessary condition, we find that $t$ must be congruent to $\frac{1}{2}n(n-5)$ modulo 2 for nonorientable quadrangular embeddings, and modulo 4 for orientable quadrangular embeddings. Liu, Ellingham, and Ye \cite{Liu-Minimal} proved that if $t \leq n-4$, an $(n,t)$-quadrangulation is necessarily minimal. They went on to find such a quadrangulation for every surface except for the sphere and the Klein bottle:

\begin{theorem}[Liu \emph{et al.}\ \cite{Liu-Minimal}]
For all integers $n$ and $t$ such that $n \geq 5$, $0 \leq t \leq n-4$, and $t \equiv \frac{1}{2}n(n-5) \pmod{4}$, there exists an orientable $(n, t)$-quadrangulation. There is also an orientable $(4,2)$-quadrangulation.
\label{thm-orient}
\end{theorem}
\begin{theorem}[Liu \emph{et al.}\ \cite{Liu-Minimal}]
For all integers $n$ and $t$ such that $n \geq 4$, $n \neq 5$, $0 \leq t \leq n-4$, and $t \equiv \frac{1}{2}n(n-5) \pmod{2}$, there exists a nonorientable $(n, t)$-quadrangulation. There is also a nonorientable $(6,3)$-quadrangulation.
\label{thm-nonorient}
\end{theorem}

An embedding is said to be \emph{face-simple} if its dual is simple, i.e., no two faces share more than one edge, and no edge has two incidences with the same face. For quadrangular embeddings, there are only a few ways that two faces can meet at two different edges, yielding a fairly general sufficient condition:
\begin{proposition}[Liu \emph{et al.}\ \cite{Liu-Minimal}, Observation 2.3]
Every orientable quadrangular embedding of minimum degree at least 3 is face-simple.
\label{prop-deg3}
\end{proposition}
Consequently, Theorem \ref{thm-orient} can automatically be strengthened to include face-simplicity for each orientable minimal quadrangulation, except $(n,t) = (4,2)$. On the sphere, Hartsfield and Ringel \cite{Hartsfield-Orientable} observed that its smallest face-simple quadrangulation is the cube graph on $8$ vertices, since its dual, the octahedral graph, is the smallest 4-regular simple planar graph. 

Liu \emph{et al.}\ \cite{Liu-Minimal} conjectured that a similar sharpening is possible for nonorientable surfaces. Hartsfield and Ringel \cite{Hartsfield-Nonorientable} showed that the smallest face-simple quadrangulation of the projective plane has 6 vertices, and their minimal $(6,3)$-quadrangulation in the Klein bottle is face-simple. The purpose of this note is to prove that the conjecture of Liu \emph{et al.} is true for all other nonorientable surfaces. We also simplify the proof of Theorem \ref{thm-orient} using the same method.

In earlier work, Liu \emph{et al.}\ \cite{EvenMapColor} (a superset of the authors of \cite{Liu-Minimal}) found an inductive construction for minimum genus even embeddings (i.e., every face has even length) of the complete graphs $K_n$ using a clever application of Bouchet's \cite{Bouchet-Diamond} diamond sum operation. When the embeddings are quadrangular, these are $(n,0)$-quadrangulations. Liu \emph{et al.}\ \cite{Liu-Minimal} contrast this construction with their proof of Theorems \ref{thm-orient} and \ref{thm-nonorient} using what they call Hartsfield's ``diagonal technique.'' However, they did not comment on the feasibility of a diamond sum approach, like the one we take here, to constructing other, non-complete minimal quadrangulations.

\section{The diamond sum operation}

Let $\psi\colon G \to S$ and $\psi'\colon G' \to S'$ be two quadrangular embeddings in disjoint surfaces $S$ and $S'$, and let $v$ and $v'$ be vertices of $G$ and $G'$, respectively, that have the same degree. Following Figure \ref{fig-diamond}(a), let $D$ be a closed disk in the surface $S$ intersecting the embedded graph only at $v$ and its incident edges, such that the neighbors of $v$ are on the boundary of $D$. Let $D'$ be defined similarly for $S'$ and $v'$. Delete the interiors of $D$ and $D'$ and identify the resulting boundaries, like in Figure \ref{fig-diamond}(b), so that each vertex on the boundary of $D$ is identified with a unique vertex on the boundary of $D'$. We call the resulting quadrangular embedding a \emph{diamond sum of $\psi$ and $\psi'$ at $v$ and $v'$}. Since the diamond sum operation takes the connected sum of the surfaces $S$ and $S'$, the resulting embedding is orientable if and only if both $\psi$ and $\psi'$ are orientable. 

\begin{figure}[t]
\centering
    \begin{subfigure}[b]{0.49\textwidth}
    \centering
        \includegraphics[scale=1]{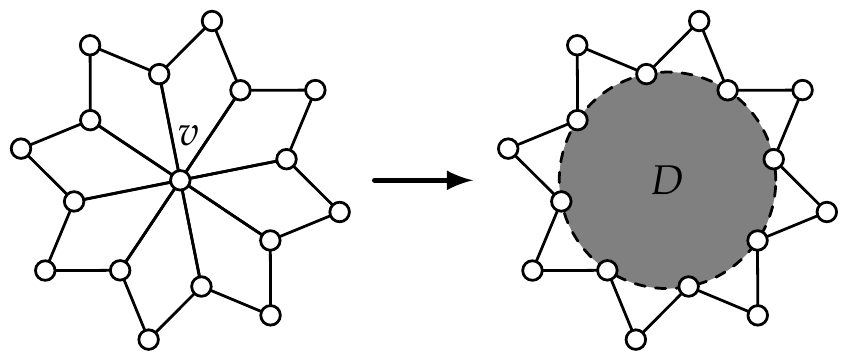}
        \caption{}
        \label{subfig-o34}
    \end{subfigure}
    \begin{subfigure}[b]{0.49\textwidth}
    \centering
        \includegraphics[scale=1]{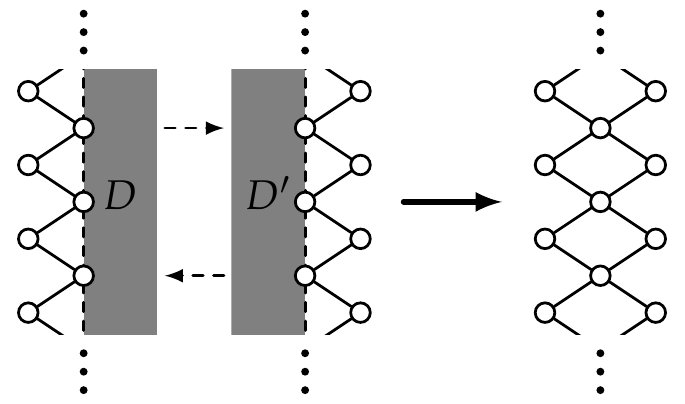}
        \caption{}
        \label{subfig-o40}
    \end{subfigure}
\caption{Excising disks (a) and merging their boundaries (b).}
\label{fig-diamond}
\end{figure}

The aforementioned proof by induction for the complete graphs uses a quadrangular embedding of $K_7$ with a subdivided edge (e.g., the embedding $\phi_{7,0}^+$ in Figure \ref{fig-embeddings}).  Any embedding of a graph with a vertex of degree 2, like $\phi_{7,0}^+$, cannot be face-simple because its two incident faces share two edges. However, Liu \emph{et al.}\ \cite{EvenMapColor} showed that applying a diamond sum near this vertex can remove the double incidence. We call an embedding \emph{nearly face-simple except at vertex $v$} if every pair of distinct faces shares at most one edge not incident with $v$. Note that face-simple embeddings trivially satisfy this property for any of its vertices. The following technical lemma guarantees face-simplicity, which we rephrase slightly for quadrangular embeddings:

\begin{lemma}[Liu \emph{et al.}\ \cite{EvenMapColor}, Lemma 3.5]
Suppose $\psi$ is a face-simple quadrangular embedding of a simple graph $G$ of minimum degree at least 3, and let $v$ be a vertex of $G$ whose neighbors are independent. Suppose $\psi'$ is a quadrangular embedding of a simple graph that is nearly face-simple except at some vertex $v'$. If the degrees of $v$ and $v'$ are the same, then any diamond sum of $\psi$ and $\psi'$ at $v$ and $v'$ is quadrangular and face-simple.
\label{lem-technical}
\end{lemma}

\section{Nonorientable surfaces}

Ringel \cite{Ringel-Orientable} showed that $K_{m,n}$ has an orientable quadrangular embedding if $m \equiv 2 \pmod{4}$, for all $n \geq 2$. When $\min(m,n) \geq 3$, these embeddings are guaranteed to be face-simple by Proposition \ref{prop-deg3}. The other quadrangular embeddings we use are shown in Figure \ref{fig-embeddings} as polygons whose sides are identified. Each embedding $\phi_{n,t}$ is of a graph homeomorphic to a simple graph on $n$ vertices and $\binom{n}{2}-t$ edges, and is either face-simple, or nearly face-simple except at vertex $x$. The ``$+$'' superscript indicates that one edge has been subdivided, and the ``$*$'' superscript indicates that the embedding is orientable. One can check that the other embeddings are nonorientable by finding two sides of the polygon whose identification creates a M\"obius strip.

The main building block in our inductive construction generalizes Lemma 4.2 of Liu \emph{et al.}\ \cite{EvenMapColor}, and it has a very similar proof:

\begin{figure}[t]
\centering
\includegraphics[scale=1]{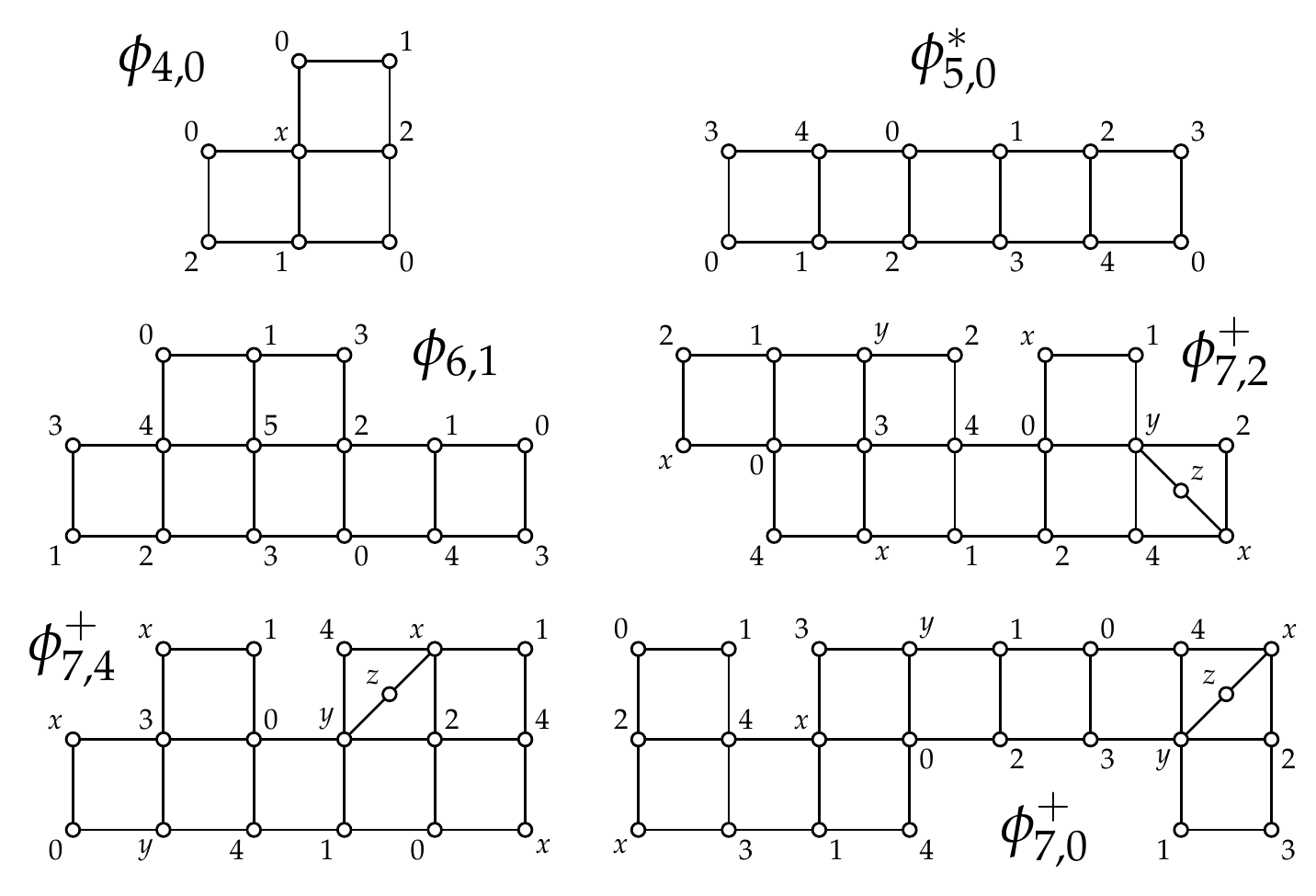}
\caption{``Unfoldings'' of quadrangular embeddings.}
\label{fig-embeddings}
\end{figure}

\begin{lemma}
Suppose, for some $n \geq 4$, there exists an $(n,t)$-quadrangulation $\psi$ of a simple graph $G$ that is nearly face-simple except at $v$, where $v$ is some vertex adjacent to every other vertex. Then, for each $i \in \{0, 2, 4\}$, there exists a face-simple nonorientable $(n+4, t+i)$-quadrangulation which has a vertex adjacent to every other vertex.
\label{lem-induct}
\end{lemma}
\begin{proof}
Consider any $\phi_{7, i}^+$ in Figure \ref{fig-embeddings}. Using edge-complements, graph joins, and disjoint unions, the embedded graph can be expressed as $\overline{K_2} + ((K_1 + H_i) \cup K_1)$, where $H_i$ is a graph formed by taking the complete graph $K_4$ and deleting $i$ edges. The vertices of $K_1+H_i$ are the numbered vertices (in particular, vertex 0 is the $K_1$), the vertices of $\overline{K_2}$ are $x$ and $y$, and the remaining $K_1$ is $z$. The missing edges in $H_2$ are $(1,3)$ and $(2,3)$, and the missing edges in $H_4$ are $(1, 2)$, $(1, 3)$, $(2, 3)$, and $(3,4)$. 

Let $\phi'$ be an orientable quadrangular embedding of $K_{6,n-1}$, and let $u$ be any vertex of degree 6. Take any diamond sum of $\phi_{7,i}^+$ and $\phi'$ at $x$ and $u$. The resulting embedding is of the graph $((K_1 + H_i) \cup K_1) + \overline{K_{n-1}}$, which is face-simple by Lemma \ref{lem-technical}. A diamond sum of this embedding and $\psi$ at $z$ and $v$ results in a quadrangular embedding of $(K_1 + H_i)+(G-v)$. In total, there are $(1+4)+(n-1)=n+4$ vertices and $i+t$ missing edges. The final embedding is face-simple (once again, by Lemma \ref{lem-technical}) and nonorientable (since $\phi_{7,i}^+$ is nonorientable), and vertex $0$ is adjacent to every other vertex.
\end{proof}

At this point, we are ready to prove a strengthening of Theorem \ref{thm-nonorient}:

\begin{theorem}
For all integers $n$ and $t$ such that $n \geq 6$, $0 \leq t \leq n-4$, and $t \equiv \frac{1}{2}n(n-5) \pmod{2}$, there exists a face-simple nonorientable $(n, t)$-quadrangulation which has a vertex adjacent to every other vertex.
\label{thm-main}
\end{theorem}
\begin{proof}
We induct on $n$, but to avoid listing out too many specific embeddings, we include $\phi_{4,0}$ and $\phi_{5,0}^*$ from Figure \ref{fig-embeddings} as base cases. Even though they are either orientable (in the case of $\phi_{5,0}^*$) or not face-simple (in the case of $\phi_{4,0}$), they still satisfy the conditions of Lemma \ref{lem-induct}. The remaining base cases are $\phi_{6,1}$, and the $(7,3)$- and $(7,1)$-quadrangulations resulting from deleting the vertex of degree 2 from $\phi_{7,2}^+$ and $\phi_{7,0}^+$, respectively.

For any integer $n \geq 8$, the permissible values of $t$ for both $n$ and $n-4$ are of the same parity. Let $n \geq 8$ and $t \geq 0$ satisfy the above hypotheses. We apply Lemma \ref{lem-induct} on an $(n-4, t-i)$-quadrangulation, where 
$$i = \begin{cases} 0 & \text{if~} t = 0,1 \\ 2 & \text{if~} t = 2,3 \\ 4& \text{if~} t \geq 4. \end{cases}$$
One can check that in all cases, $t-i \leq (n-4)-4$, so the desired quadrangulation exists by induction.
\end{proof}

\section{Orientable surfaces}

We give a new proof of Theorem \ref{thm-orient} along the same lines as our previous construction for the nonorientable case. By Proposition \ref{prop-deg3}, there is no need to worry about face-simplicity here. The inductive construction requires more base cases than before, but once again, we are able to reduce the set of \emph{ad hoc} embeddings by showing that some of these base cases can themselves be built from smaller embeddings. First, consider the embedding $\phi_{11,8}^{+*}$ in Figure \ref{fig-orient}. It is missing the edges of two 4-cycles $(1 \,\, 2 \,\, 3 \,\, 4)$ and $(5 \,\, 6 \,\, 7 \,\, 8)$. These edges can be incorporated into the embedding by applying the handle augmentation described in Figure \ref{fig-handle} to the pairs of faces labeled $\alpha$ and $\beta$ in Figure \ref{fig-orient}, resulting in the embeddings $\phi_{11,4}^{+*}$ and $\phi_{11,0}^{+*}$. Like before, these embeddings drive the inductive step of the proof:

\begin{figure}[t]
\centering
\includegraphics[scale=1]{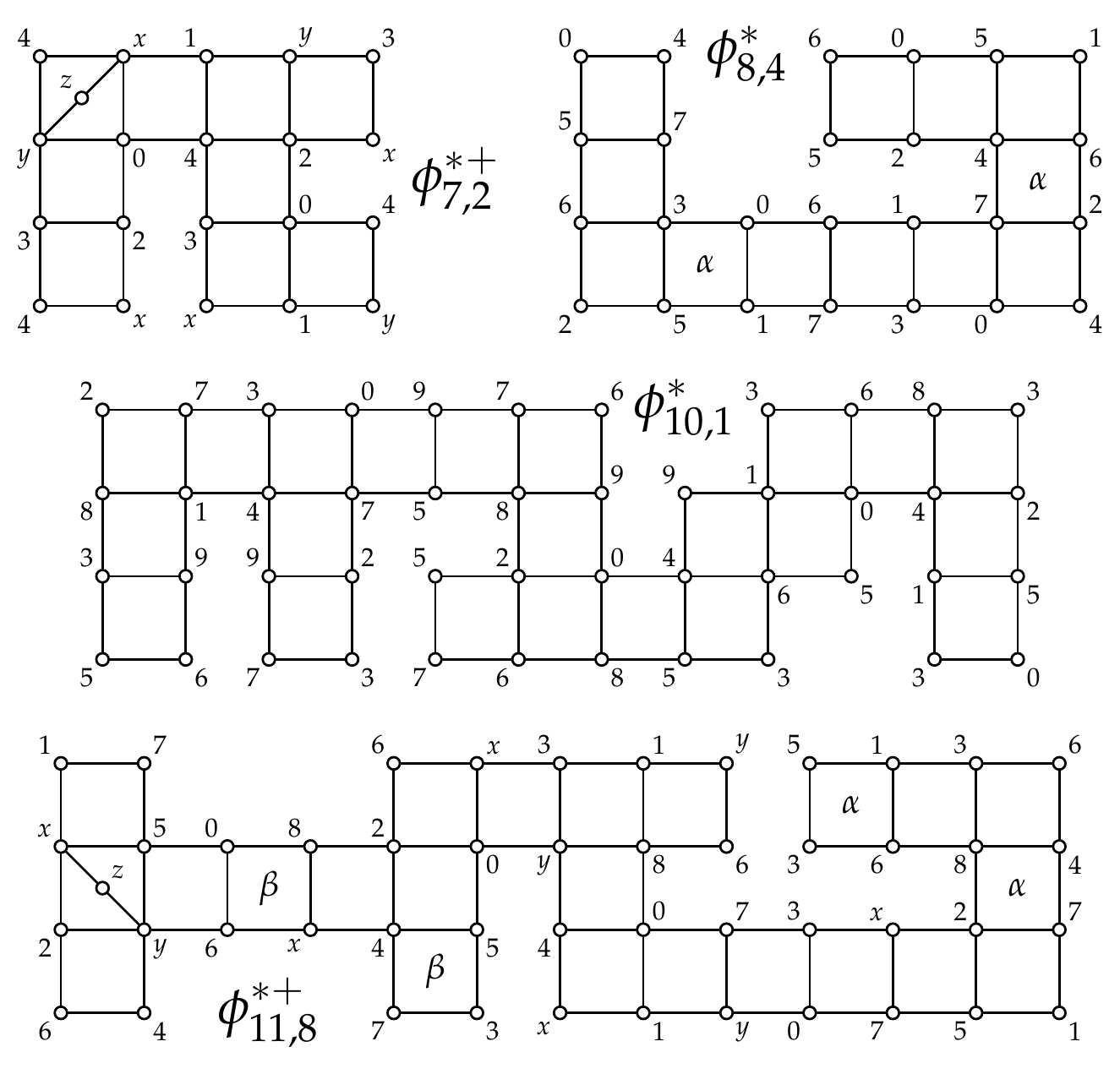}
\caption{More embeddings for the orientable analogue.}
\label{fig-orient}
\end{figure}

\begin{figure}[t]
\centering
\includegraphics[scale=1]{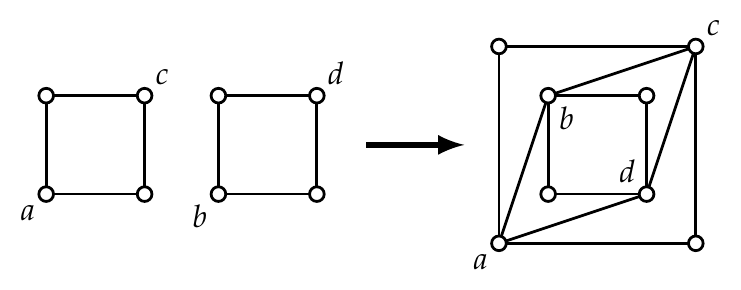}
\caption{Adding the edges of the 4-cycle $(a \,\, b \,\, c \,\, d)$.}
\label{fig-handle}
\end{figure}

\begin{lemma}
Let $\psi$ be an orientable $(n,t)$-quadrangulation with a vertex adjacent to every other vertex.  Then, for each $i \in \{0, 4, 8\}$, there exists an orientable $(n+8, t+i)$-quadrangulation which has a vertex adjacent to every other vertex. 
\label{lem-orientinduct}
\end{lemma}
\begin{proof}
Each $\phi_{11,i}^{+*}$ is isomorphic to $\overline{K_2} + ((K_1 + J_i) \cup K_1)$, where $J_i$ is a graph formed by taking $K_8$ and deleting $i$ edges. Similar to how we proved Lemma \ref{lem-induct}, we take diamond sums of the embeddings $\phi_{11,i}^{+*}$, $K_{10,n-1}$, and $\psi$.
\end{proof}

We also have an orientable embedding $\phi_{7,2}^{+*}$ in Figure \ref{fig-orient}, which yields a partial analogue of Lemma \ref{lem-induct}:

\begin{proposition}
Let $\psi$ be an orientable $(n,t)$-quadrangulation with a vertex adjacent to every other vertex.  Then, there exists an orientable $(n+4, t+2)$-quadrangulation which has a vertex adjacent to every other vertex. 
\label{prop-intermediate}
\end{proposition}

While this will not be used for the inductive step, it helps in finding some of the other base cases. We use it in proving a marginally stronger version of Theorem \ref{thm-orient}:

\begin{theorem}
For all integers $n$ and $t$ such that $n \geq 5$, $0 \leq t \leq n-4$, and $t \equiv \frac{1}{2}n(n-5) \pmod{4}$, there exists an orientable $(n, t)$-quadrangulation which has a vertex adjacent to every other vertex.
\label{thm-mainorient}
\end{theorem}
\begin{proof}
We first establish the base cases, the $(n,t)$-quadrangulations for $n = 5, \dotsc, 14$, after which Lemma \ref{lem-orientinduct} can be applied repeatedly to obtain all higher-order minimal quadrangulations. For $n = 5$, we reuse $\phi_{5,0}^*$ from Figure \ref{fig-embeddings}. For $n=6$, there are no permissible values of $t$. For $n = 7, 11$, we repurpose the embeddings $\phi_{7,2}^{*+}$, $\phi_{11,0}^{*+}$, and $\phi_{11,4}^{*+}$ by deleting their degree 2 vertices. For $n = 8$, the handle augmentation in Figure \ref{fig-handle} can also be used on $\phi_{8,4}^*$ to obtain an $(8,0)$-quadrangulation. For $n = 10$, $\phi_{10,1}^*$ in Figure \ref{fig-orient} is a $(10,1)$-quadrangulation. To construct a $(10,5)$-quadrangulation, we add a degree 2 vertex inside one of the faces of $\phi_{5,0}^*$ to obtain a (non-minimal) $(6,3)$-quadrangulation, and then apply Proposition \ref{prop-intermediate}. The remaining cases $n = 9, 12, 13, 14$ are achieved by applying either by Lemma \ref{lem-orientinduct} or Proposition \ref{prop-intermediate} to some of the aforementioned embeddings. All of these embeddings have a vertex adjacent to every other vertex. 

For all $n \geq 15$, we apply Lemma \ref{lem-orientinduct} on an $(n-8, t-i)$-quadrangulation, where 
$$i = \begin{cases} 0 & \text{if~} t = 0,1,2,3 \\ 4 & \text{if~} t = 4,5,6,7 \\ 8 & \text{if~} t \geq 8. \end{cases}$$
\end{proof}

\bibliographystyle{alpha}
\bibliography{biblio}

\end{document}